\documentclass[11pt,oneside]{article}
\usepackage[utf8]{inputenc}

\usepackage[english]{babel}
\usepackage[a4paper,scale=0.73]{geometry}
\usepackage{indentfirst}

\usepackage{amsmath,amssymb,amsfonts,amsthm,mathrsfs,stmaryrd}
\usepackage{graphicx,caption,float}
\usepackage[usenames, dvipsnames]{xcolor}
\usepackage{yhmath}
\usepackage[linesnumbered, ruled, vlined]{algorithm2e}
\usepackage{esvect}
\usepackage{longtable}

\usepackage[hang,flushmargin]{footmisc} 

\usepackage{enumitem}
\usepackage{multirow}

\usepackage{caption}
\usepackage{yhmath}

\usepackage[colorlinks=true, pdfstartview=FitV,linkcolor=ForestGreen,citecolor=ForestGreen, urlcolor=blue]{hyperref}

\author{
Ramona Anton$^{\small 1}$,
Nicolae Mihalache$^{\small 2}$
and
Fran\c{c}ois Vigneron$^{\small 3}$
}

\newcommand{\authorramona}{{\small
\textbf{Ramona Anton.}\\
Sorbonne Universit\'{e}, IMJ-PRG, CNRS UMR 7586, F-75252 Paris.
\phantom{x}\hfill\texttt{ramona.anton@imj-prg.fr}
}}

\newcommand{\authornicu}{{\small
\textbf{Nicolae Mihalache.}\\
Universit\'{e} Paris-Est Creteil, CNRS UMR 8050, LAMA, F-94010 Creteil, France and
Universit\'{e} Gustave Eiffel, LAMA, F-77447 Marne-la-Vallée.
\phantom{x}\hfill\texttt{nicolae.mihalache@u-pec.fr}
}}

\newcommand{\authorfv}{{\small
\textbf{François Vigneron.}\\
Universit\'{e} de Reims Champagne-Ardenne,
Laboratoire de Mathématiques de Reims,
UMR 9008 CNRS, Moulin de la Housse, BP 1039,
F-51687 Reims.
\phantom{x}\hfill\texttt{francois.vigneron@univ-reims.fr}
}}

\newtheorem{thm}{Theorem}

\newcommand{\N}{\mathbb{N}}
\newcommand{\Z}{\mathbb{Z}}

\newcommand{\C}{\mathbb{C}}

\DeclareMathOperator{\crit}{Crit}

\newcommand{\eg}[1][~]{\textit{e.g.}#1}
\newcommand{\ie}[1][~]{\textit{i.e.}#1}

\newcommand{\zeros}{\mathcal{Z}}
\newcommand{\cv}{\mathcal{V}}
\DeclareMathOperator{\crits}{Crit}

\begin{document}
\title{A short ODE proof of the Fundamental Theorem of Algebra}
\date{}
\maketitle

\begin{abstract}  
We propose a short proof of the Fundamental Theorem of Algebra
based on the ODE that describes the Newton flow and the fact
that the value $|P(z)|$ is a Lyapunov function. It clarifies an idea that goes back to Cauchy.
\end{abstract}

The Fundamental Theorem of Algebra is one of the major results of mathematics;
historical reviews can be found~\eg in~\cite{Sma81}, \cite[Chap.~II]{Dieu86} or \cite{GTM123}.
In practice, the roots of polynomials are computed with a variety of different algorithms (see~\eg\cite{MV2023}
for an exhaustive review) and, notably, with Newton-Raphson's method~\cite{Gal00}.
It is well known that the different basins of attraction of the iterative Newton method are deeply intricate.
The existence of universal sets of starting points such that the closure of all Newton's iterates will always contain
all roots is remarkable though highly non-trivial~(see Hubbard-Schleicher-Sutherland \cite{HSS01}).
On the other hand, the flow of Newton's ODE presented below in~\eqref{eq:newtonFlow} is topologically much simpler
and suitable for a short constructive proof of the fundamental Theorem of Algebra that we present in this note.

\medskip
This idea can be found, at a higher level of generality, in Hirsch-Smale~\cite{HS1979}.
However, the proof below is of great practical importance because, in addition to the general considerations of~\cite{HS1979}, we
are able to control very explicitly the speed of convergence of each Newton trajectory.
For example,~\eqref{eq:Lyapunov} implies that, away from roots and critical points, each step of Newton's method approximately divides
the value of the polynomial by~$e$ (see~Figure~\ref{fig}), which is a property that, to the best of our knowledge,
had not been noticed before. 
In~\cite{MV2023} we use this idea to split a non-trivial polynomial of record degree~$10^{12}$.
We believe it is important to highlight it.

\bigskip

Given a polynomial $P\in\C[z]$ of degree $d\geq1$, we denote by $\zeros(P)=P^{-1}(0)$ the set of its zeros
and by $\crits(P)=\zeros(P')$ the set of its critical points, a finite set (by Euclidean division). Let also $\cv(P)=P(\crits(P))$ the 
set of critical values of $P$. 

\medskip
The \textit{Newton flow} of $P$ is defined by the ODE
\begin{equation}\label{eq:newtonFlow}
\varphi_z(0)=z \quad\text{and}\quad \varphi_z'(t) = -\frac{P(\varphi_z(t))}{P'(\varphi_z(t))}\cdotp
\end{equation}
The set $\zeros(P)$ is composed of stationary solutions of~\eqref{eq:newtonFlow}.
Thanks to the Cauchy-Lipschitz theorem, the Newton flow is locally well defined starting from any point~$z\notin\crits(P)$
and can be extended locally as long as $\varphi_z(t)\notin\crits(P)$. For $z\notin\crits(P)$, let us denote by $T(z)\in(0,+\infty]$ the
maximal forward time of existence of $\varphi_z$.
One has
\begin{equation}\label{eq:Lyapunov}
\forall  z\notin\crits(P), \quad\forall t\in [0,T(z))\qquad
P(\varphi_z(t)) = e^{-t} P(z) 
\end{equation}
because Leibniz's rule and~\eqref{eq:newtonFlow} imply
\[
\frac{d}{dt} \left[ e^t P(\varphi_z(t)) \right] = e^t P(\varphi_z(t)) + e^t P'(\varphi_z(t)) \varphi_z'(t)=0.
\]
In particular, from~\eqref{eq:Lyapunov} we see that Newton's flow is \textit{iso-angle}
\ie $\arg P(\varphi_z(t)) = \arg P(z)$ if $P(z)\neq 0$. Moreover, for $t \in (0,T(z))$,
$\varphi_z(t) \in \mathcal{B}(z)= \{ y\in \C \,;\, |P(y)|\leq |P(z)|\}$, which
is a compact set because $\lim\limits_{|z|\to\infty} |P(z)| = +\infty$.

\begin{thm}
Every non-constant polynomial $P  \in \C[z]$ has at least one root in $\mathbb{C}$.
\end{thm}

\begin{proof}
Our strategy to find a root of $P$ is to choose a starting point $z_0$ with $T(z_0)=+\infty$ and then pass to the limit in~\eqref{eq:Lyapunov}. We will see that it is enough that $\arg(P(z_0))$ avoids the finite set $\arg(\cv(P))$. Observe that $0 \notin \cv(P)$, otherwise $P$ has a root in $\crits(P)$ and we stop here.

\medskip
If $T(z)<\infty$ for some $z \in \C  \setminus \crits(P)$, then $P(z)\neq 0$. As $t\to T(z)$,  the orbit $\varphi_z(t)$ leaves any compact subset of $\C\backslash\crit(P)$. As $\varphi_z(t) \in \mathcal{B}(z)$, a compact set, there is a sequence~$t_n\to T(z)$ such that $\varphi_z(t_n)$ converges to some point $c \in \crits(P)$.
The continuity of $P$ at $c$ implies that  $\arg P(z) = \arg P(\varphi_z(t_n))=\arg(P(c)) \in  \arg(\cv(P))$.

\medskip
Let us find $z_0$ with $\arg P(z_0) \notin \arg(\cv(P))$. As $\lim\limits_{|z|\rightarrow \infty} z^{-d}P(z) = \alpha \in \C^*$, taking  $z_0=r e^{i\vartheta}$ with
$d\vartheta +  \arg\alpha\notin \arg(\cv(P)) \mod 2\pi$  
and $r>0$ large enough ensures that $\arg(P(z_0)) \notin \arg(\cv(P))$.
Therefore $T(z_0)=+\infty$ and the bounded sequence $\left(\varphi_{z_0}(n)\right)_{n\in\N} \subset \mathcal{B}(z_0)$ admits an accumulation point $z_\ast$.
The continuity of $P$ at $z_\ast$ and~\eqref{eq:Lyapunov} imply that $P(z_\ast)=\lim\limits_{n\to\infty} e^{-n}P(z_0) = 0$.
\end{proof}

\paragraph{Acknowledgement.} The second author would like to thank Dierk Schleicher for several insightful discussions about the Newton  method and the Newton flow.


\begin{figure}[H]
\captionsetup{width=\linewidth}
\begin{center}
\includegraphics[width=\textwidth]{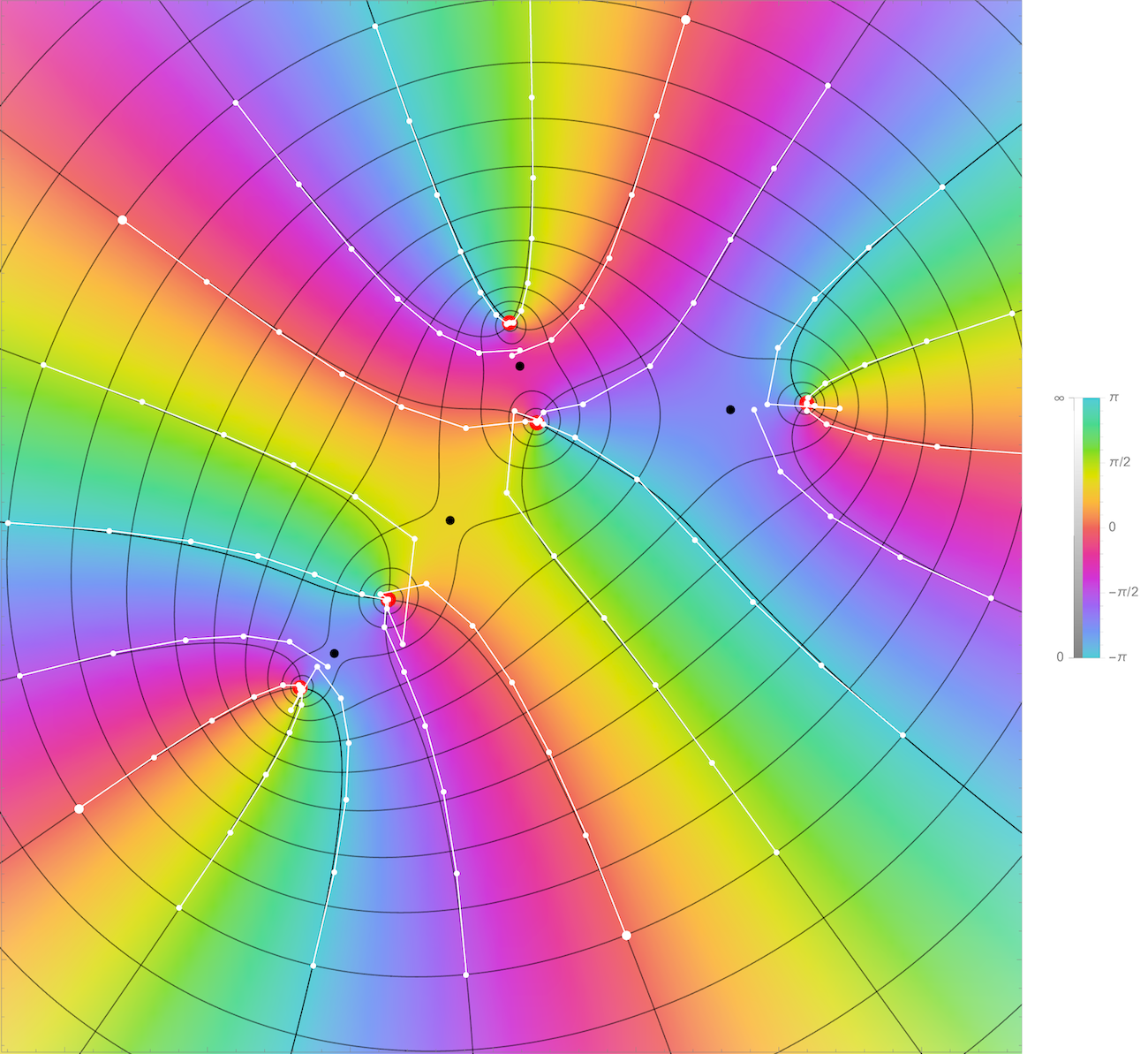}
\caption{\label{fig}
Polynomial of degree 5, with roots (red) and critical points (black). The level lines $p(z)=2^n$ with $n\in\Z$
and iso-angles (solution of Newton's flow $z'=-p(z)/p'(z)$ starting from the points $z_0$ such that
$p(z_0)\in 2^{n_0}\{\pm 1, \pm i\}$ for some $n_0\in\Z$) are indicated (gray).
The Newton's steps (white) follow approximately Newton's flow, away from the roots and the critical points.
}
\end{center}
\end{figure}

\vspace*{1em}\noindent
$^{\small 1}$ \authorramona\\[1ex]
$^{\small 2}$ \authornicu\\[1ex]
$^{\small 3}$ \authorfv

\end{document}